\documentclass{amsart}

\usepackage{amssymb}
\usepackage[all,cmtip]{xy}
\usepackage{tikz-cd}
\usetikzlibrary{matrix}
\usepackage{hyperref}
\usepackage{graphicx}

\newcommand{\injto}{\hookrightarrow}

\newcommand{\longto}{\longrightarrow}

\makeatletter
\newcommand*\rel@kern[1]{\kern#1\dimexpr\macc@kerna}
\newcommand*\widebar[1]{%
  \begingroup
  \def\mathaccent##1##2{%
    \rel@kern{0.8}%
    \overline{\rel@kern{-0.8}\macc@nucleus\rel@kern{0.2}}%
    \rel@kern{-0.2}%
  }%
  \macc@depth\@ne
  \let\math@bgroup\@empty \let\math@egroup\macc@set@skewchar
  \mathsurround\z@ \frozen@everymath{\mathgroup\macc@group\relax}%
  \macc@set@skewchar\relax
  \let\mathaccentV\macc@nested@a
  \macc@nested@a\relax111{#1}%
  \endgroup
}
\makeatother

\makeatletter
\def\@citex[#1]#2{\leavevmode
  \let\@citea\@empty
  \@cite{\bfseries\@for\@citeb:=#2\do
    {\@citea\def\@citea{,\penalty\@m\ }%
     \edef\@citeb{\expandafter\@firstofone\@citeb\@empty}%
     \if@filesw\immediate\write\@auxout{\string\citation{\@citeb}}\fi
     \@ifundefined{b@\@citeb}{\hbox{\reset@font\bfseries ?}%
       \G@refundefinedtrue
       \@latex@warning
         {Citation `\@citeb' on page \thepage \space undefined}}%
       {\@cite@ofmt{\csname b@\@citeb\endcsname }}}}{\mdseries #1}}
\makeatother

\DeclareMathOperator\Sym{Sym}

\DeclareMathOperator\End{End}

\DeclareMathOperator\im{im}

\DeclareMathOperator\cone{cone}

\DeclareMathOperator\Fil{Fil}
\DeclareMathOperator\gr{gr}

\DeclareMathOperator\Tot{Tot}

\DeclareMathOperator\Simp{Simp}
\DeclareMathOperator\HoSimp{HoSimp}

\DeclareMathOperator\Ch{Ch}
\DeclareMathOperator\cross{cr}

\def\bQ{{\mathbf{Q}}} \def\bZ{{\mathbf{Z}}}

\def\cO{{\mathcal{O}}}

  \def\cS{{\mathcal{S}}}
\def\cB{{\mathcal{B}}}  \def\cP{{\mathcal{P}}}
\def\cA{{\mathcal{A}}} \def\cK{{\mathcal{K}}}

  \def\rH{{\rm H}} 
\def\rL{{\mathrm L}} \def\rK{{\rm K}}

\def\id{{\rm id}}

\theoremstyle{plain}
\newtheorem{theorem}{Theorem}
\newtheorem{corollary}{Corollary}

\newtheorem{proposition}{Proposition}
\theoremstyle{definition}
\newtheorem{definition}{Definition}
\newtheorem{example}{Example}

\begin{document}

\title
 {Non-additive functors and Euler characteristics}

\author
{Niels {u}it de Bos}

\email{niels.uitdebos@uni-due.de}

\address{Universit\"at Duisburg-Essen,
	Fakult\"at f\"ur Mathematik, 
	Universit\"atsstra{\ss}e 2,
	45117 Essen,
	Germany}

\author
{Lenny Taelman}

\email{l.d.j.taelman@uva.nl}

\address{Korteweg-de Vries Instituut,
         Universiteit van Amsterdam,
         P.O. Box 94248, 
		1090 GE Amsterdam,
		the Netherlands}

\begin{abstract}
We show under suitable finiteness conditions that a functor between abelian categories induces a (not necessarily additive) map between their Grothendieck groups. This is related to the derived functors of Dold and Puppe, and generalizes a theorem of Dold.
\end{abstract}

%\classification{18F25, 19A99, 13D15}

% Leave these items like this, and the journal will fill them in.
%\received{Month Day, Year}   % receive date (for example: October 11, 1999)
%\revised{Month Day, Year}    % date of revision; omit, if no revision;
%                             % if multiple revisions, separate by commas
%\published{Month Day, Year}  % publish date
%\submitted{Daniel Dugger}      % Name of Journal's Editor, who handled Article 
%\volumeyear{2014} % Volume Year
%\volumenumber{16} % Volume Number 
%\issuenumber{2}   % Issue Number
%\startpage{1}     % PageNumber of first page
%\articlenumber{1} % Sequence number of article within issue

\maketitle

\section{Introduction}

Let $\cA$ and $\cB$ be abelian categories. Assume that $\cA$ has enough projectives. If $F\colon \cA \to \cB$ is a functor, not necessarily additive, then we denote by 
\[
	\rL F\colon \cK_{\geq 0}(\cA) \to \cK_{\geq 0}(\cB),
\]
the total derived functor of Dold and Puppe \cite{DoldPuppe61}, where $\cK_{\geq 0}$ is the homotopy category of chain complexes concentrated in non-negative degrees. The definition of $\rL F$ involves taking a projective resolution, using the Dold-Kan theorem to pass to an associated simplicial object, applying $F$ and taking the associated chain complex. The details will be recalled in Section \ref{sec:DoldPuppe}. If $F$ is additive, then $\rL F$ coincides with the usual total left  derived functor. Note that our notation is slightly non-standard, in that we denote by $\rL F$ the functor between homotopy categories, instead of the induced functor between derived categories. 

We will require that $F(0)=0$ and that $F$ is of \emph{degree $\leq d$} for some positive integer $d$. The definition is recursive. $F$ is said to be of degree $\leq 1$ if it is additive, and of degree $\leq d$ if and only if there exists a functor $G\colon \cA\times \cA \to \cB$
of degree $\leq d-1$ in both arguments, together with a functorial decomposition
\[
	F(X \oplus Y) = F(X) \oplus F(Y) \oplus G(X,Y).
\]
Examples of functors of degree $\leq d$ are the Schur functors $\wedge^d$ and $\Sym^d$.

Let $\cA_0$  be a weak Serre subcategory of $\cA$ (this is a slight generalization of the more common notion of a Serre subcategory, see see \cite[Tag 02MN]{stacks-project}). Let 
$\cK_{\geq 0}^{\cA_0}(\cA)$ be the full subcategory of $\cK_{\geq 0}(\cA)$ consisting of those complexes $X_\bullet$ such that
\begin{enumerate}
\item $\rH_i(X_\bullet)=0$ for $i \gg 0$;
\item $\rH_i(X_\bullet) \in \cA_0$ for all $i$.
\end{enumerate} 
Every object $X_\bullet \in \cK_{\geq 0}^{\cA_0}(\cA)$ has an Euler characteristic
\[
	\chi(X_\bullet) = \sum_i (-1)^i \big[ \rH_i(X_\bullet) \big] 
\]
in the Grothendieck group $\rK_0(\cA_0)$. Let $\cB_0$ be a weak Serre subcategory of $\cB$, and define $\cK_{\geq 0}^{\cB_0}(\cB)$ analogously. 

Our main result is the following theorem.

\begin{theorem}\label{thm:one}
Assume that $F\colon \cA \to \cB$ is of finite degree $\leq d$ with $F(0)=0$, and that $\rL F$ maps
$\cK_{\geq 0}^{\cA_0}(\cA)$ to $\cK_{\geq 0}^{\cB_0}(\cB)$. Then
there is a unique map $f\colon \rK_0(\cA_0) \to \rK_0(\cB_0)$ such that the square
\begin{equation}\label{eq:square-one}
\begin{tikzcd}
\cK_{\geq 0}^{\cA_0}(\cA) \arrow{r}{\rL F} \arrow{d}{\chi}
& \cK_{\geq 0}^{\cB_0}(\cB) \arrow{d}{\chi} \\
 \rK_0(\cA_0) \arrow{r}{f} 
& \rK_0(\cB_0)
\end{tikzcd}
\end{equation}
commutes. Moreover, the function $f$ is of degree $\leq d$.
\end{theorem}

Here $f$ being of degree $\leq d$ is also defined recursively: $f$ is of degree $\leq 1$ if it is linear, and of degree $\leq d$  if the function
\[
	g(x,y) := f(x+y) - f(x) - f(y) 
\]
is of degree $\leq d-1$ in both arguments.

Typically, $\rK_0(\cA_0)$ is only interesting if $\cA_0$ is small enough, whereas
$\cA$ needs to be large enough to contain projective resolutions (or with the dual version of the theorem, injective resolutions) of objects in $\cA_0$. For example $\cA_0$ could be the category of finite abelian groups, or of coherent $\cO_X$-modules on a scheme $X$, with $\cA$ the category of finitely generated abelian groups respectively quasi-coherent $\cO_X$-modules. In both cases the natural map $\rK_0(\cA_0) \to \rK_0(\cA)$ is the zero map.  

If $\cA_0=\cA$ and if every object of $\cA$ has a finite projective resolution, then Theorem \ref{thm:one} is a theorem of Dold \cite{Dold72}. Our proof is similar in flavour to Dold's, but because we cannot just compute Euler characteristics in $\rK_0(\cA)$, 
we need more refined  constructions that  only involve chain complexes (or simplicial objects) in $\cA$ with homology in $\cA_0$. A crucial ingredient in our proof is a double complex due to K\"ock \cite{Kock01}, which induces a functorial resolution of $F(P/Q)$ for a quotient $P/Q$ in terms of cross effects of $F$ applied to $P$ and $Q$, see \S \ref{sec:Koeck-resolution}.

\bigskip

We end this introduction with three examples where we compute the map $f$ of the theorem explicitly.

\begin{example}
If $F\colon \cA \to \cB$ is \emph{additive} then the theorem follows immediately from the long exact sequence of homology. The map $f$ is additive, and characterized by
\[
	f([X]) = \sum_i (-1)^i [ \rL_iF X ]
\]
in $\rK_0(\cB_0)$, for all $X \in \cA_0$. (Note that the sum is finite because of our assumption on $\rL F$).
\end{example}

\begin{example}
Let $R$ be a commutative ring and let
$\cA=\cA_0$ be the category of finitely generated $R$-modules. Assume every $M \in \cA$ has a finite projective resolution. Consider the symmetric  power 
$F=\Sym^d \colon \cA \to \cA$. Every element of $\rK_0(\cA)$ is of the form
$[P]-[Q]$ for projective objects $P$, $Q$ in $\cA$. Consider a complex $Q\overset{0}{\to} P$ in degree $1$ and $0$. 
Using a theorem of Quillen \cite[I.4.3.2]{Illusie71} one can show that
\[
	\rL_i\Sym^d(Q\to P) = \Sym^{d-i} P \otimes \wedge^i Q
\]
for every $i$ (See also \cite[2.4]{Kock01}). It follows that the map $f$ of Theorem \ref{thm:one} is 
\[
	f\colon \rK_0(\cA) \to \rK_0(\cA),\,\,
	[P]-[Q] \mapsto \sum_{i=0}^d (-1)^i \big[\Sym^{d-i} P\otimes\wedge^{i} Q  \big],
\]
and in particular that this map is well-defined. Similarly, one finds for $G=\wedge^d$ the map
\[
	g\colon \rK_0(\cA) \to \rK_0(\cA),\,\,
	[P]-[Q] \mapsto \sum_{i=0}^d (-1)^i \big[\wedge^{d-i} P\otimes\Gamma^{i} Q  \big]
\]
where $\Gamma^i Q$ denotes the $i$-th divided power module of $Q$.
\end{example}

In the third example we consider a situation where it is necessary to separate the roles of $\cA$ and $\cA_0$.

\begin{example}
Let $\cA_0=\cB_0$ be the category of finite (torsion) $\bZ$-modules, and $\cA=\cB$ the category of all $\bZ$-modules.  Consider the functors $F=\Sym^d$ and $G=\wedge^d$ from $\cA$ to $\cB$. If $X_\bullet$ is in $\cK_{\geq 0}^{\cA_0}(\cA)$ then the groups $\rL_i F(X_\bullet)$ and $\rL_i G(X_\bullet)$ are finitely generated and vanish for $i\gg 0$. Moreover, they vanish after tensoring with $\bQ$, hence $\rL F(X_\bullet)$ and $\rL G(X_\bullet)$ lie
in $\cK_{\geq 0}^{\cB_0}(\cB)$.

The cardinality of a module defines an isomorphism $\rK_0(\cA_0)\cong\bQ^\times_{>0}$. We claim that the induced maps $f,g\colon \bQ^\times_{>0} \to \bQ^\times_{>0}$  are given by
\[
	f(x)=x \quad \text{ and } g(x)=x^{(-1)^{d-1}},
\]
for all $x\in \bQ^\times$. 

Indeed: since  $f$ and $g$ are of degree $\leq d$, it suffices to show the above identities for all positive \emph{integers} $x$. Let $m$ be a positive integer and consider the complex $X$ consisting of a {cyclic} group $\bZ/m\bZ$ placed in degree $0$. Because $\wedge^i \bZ=0$ for $i\neq 0,1$ one can easily compute  $\rL_i\Sym^d X$ and $\rL_i\!\wedge^d X$ using the Koszul complexes associated to the resolution $\bZ\overset{m}{\to}\bZ$ of $X$ (see~\cite[2.4, 2.7]{Kock01}). One finds 
\[
	\rL_i \Sym^d (\bZ/m\bZ) \cong
	\begin{cases}
		\bZ/m\bZ & (i=0) \\
		0 & (i\neq 0)
	\end{cases}
\]
and
\[
	\rL_i\!\wedge^d (\bZ/m\bZ) \cong
	\begin{cases}
		0 & (i\neq d-1) \\
		\bZ/m\bZ & (i=d-1).
	\end{cases}
\]
and hence $f(m)=m$ and $g(m)=m^{(-1)^{d-1}}$ as claimed. (For an alternative computation of  $\rL_i\Sym^d(\bZ/m\bZ)$ and  $\rL_i\!\wedge^d (\bZ/m\bZ)$ see Jean \cite[\S 2.3]{Jean02}).
\end{example}

\subsection*{Acknowledgements}
The authors are grateful to Dan Dugger for pointing them to the paper of Dold \cite{Dold72}, and to Bernhard K\"ock for his many valuable comments on  earlier versions of this paper. 

This paper is an outgrowth of the master's thesis of the first author, written under supervision of the second. The second author is supported by a grant of the Netherlands Organization for Scientific Research (NWO).

\section{Cross-effects and functors of finite degree} \label{sec:cross-effects}

In this section we briefly summarize the definition and main properties of cross effect functors. We refer to the the original text of Eilenberg and Mac Lane \cite{EilenbergMacLane54} for proofs and more details.

Let $\cA$ be an abelian category. Let $X_1,\ldots, X_n \in \cA$. Put $X:=X_1\oplus\cdots \oplus X_n$.
Let $e_i\in \End X$ be the idempotent with image $X_i$. For a subset $S\subset I=\{1,\ldots, n\}$ we define
\[
	e_S := \sum_{i\in S} e_i \in \End X.
\]
This is the idempotent with image $\oplus_{i\in S} X_i \subset X$. 

Now let $\cB$ be an abelian category and let $F$ be a functor $\cA \to \cB$ with $F(0)=0$. Consider the endomorphism
\[
	\cross_n(e_1,\ldots, e_n) := \sum_{S \subset I} (-1)^{n-|S|} F(e_S) \in \End F(X).
\]
Then one verifies that $\cross_n(e_1,\ldots,e_n)$ is \emph{idempotent}.

\begin{definition}The $n$-th \emph{cross effect} of $F$ is the functor $F_n\colon \cA^n\to \cB$ given by
\[
	F_n(X_1,\ldots, X_n) :=  \im \cross_n(e_1,\ldots, e_n) \subset  F(X_1\oplus \cdots \oplus X_n).
\]
\end{definition}

Note that $F_0=0$ and $F_1=F$. We now list some more basic properties that are useful in working with the cross effect functors.

\begin{proposition} \label{prop:cross-decomposition}
There are isomorphisms
\[
	F(X_1 \oplus X_2) = F(X_1) \oplus F(X_2) \oplus F_2(X_1,X_2),
\]
functorial in $X_1$ and $X_2$. The functor $F$ is additive if and only if $F_2$ vanishes.
 \qed
\end{proposition}

The higher cross-effect functors satisfy a kind of associativity property that allows them to be computed in a recursive way:

\begin{proposition} \label{prop:cross-associative}
Let $X_1,\ldots,X_n \in \cA$. Let $G\colon \cA \to \cB$ be the functor given by
\[
	G(Y) := F_{n+1}(X_1,\ldots,X_n,Y)
\]
then there are  isomorphisms
\[
	G_m(Y_1,\ldots, Y_m) = F_{m+n}(X_1,\ldots,X_n,Y_1,\ldots, Y_m),
\]
functorial in the $X_i$ and $Y_i$. \qed
\end{proposition}

Propositions \ref{prop:cross-decomposition} and \ref{prop:cross-associative} give a decomposition
\[
	F(X_1\oplus \cdots \oplus X_n) = \bigoplus_{0<d\leq n} \bigoplus_{i_1<\cdots<i_d} 
	F_d(X_{i_1}, \ldots, X_{i_d} ).
\]

\begin{definition}Let $F\colon \cA \to \cB$ be a functor with $F(0)=0$. Let $d$ be a positive integer. We say that $F$ is \emph{of degree $\leq d$} if the functor $F_{d+1}$ vanishes.
\end{definition}

Using Propositions \ref{prop:cross-decomposition} and \ref{prop:cross-associative} one sees that this definition coincides with the recursive definition given in the introduction.

\section{K\"ock's resolution}\label{sec:Koeck-resolution}

Let $F\colon \cA \to \cB$ be a functor with $F(0)=0$. Let
\[
	0\to X\to Y\to Z \to 0
\]
be a split short exact sequence in $\cA$. Following K\"ock \cite{Kock01}, we will describe an explicit resolution of $F(Z)$ in terms of the map $X\to Y$ and the cross effect functors of $F$.  We do not choose a preferred splitting, and insist that all constructions be functorial in the short exact sequence $0\to X\to Y\to Z \to 0$.

Let $n> 0$. For $1\leq i \leq n$ consider the maps $\delta_{n,i} \colon X^{\oplus (n+1)} \to X^{\oplus n}$ given by
\[
	(x_1,\ldots, x_{n+1}) \mapsto 
	(x_1,\ldots, x_{i-1}, x_{i} + x_{i+1}, x_{i+2}, \ldots, x_{n+1} )
\]
These induce maps
\[
	F(\delta_{n,i}) \colon F(X^{\oplus (n+1)}) \to F(X^{\oplus n})
\]
which restrict to maps
\[
	F(\delta_{n,i}) \colon F_{n+1}(X,\ldots, X, X) \to F_n(X,\ldots, X).
\]
Let $d_n$ be the map
\[
	d_n = \sum_i (-1)^i F(\delta_{n,i}) \colon F_{n+1}(X,\ldots, X, X) \to F_n(X,\ldots, X).
\]
One verifies directly that 
\[
	\cdots \longto F_3(X,X,X) \overset{d_2}{\longto} F_2(X,X) \overset{d_1}{\longto} F(X)
\]
forms a complex in $\cB$. Using the map $X\to Y$, one similarly constructs a complex
\[
	\cdots \longto F_3(X,X,Y) \longto F_2(X,Y) \longto F(Y).
\]
Let $C$ be the double chain complex
\[
\begin{tikzcd}
\cdots \arrow{r} 
	& F_3(X,X,Y) \arrow{r} 
	& F_2(X,Y) \arrow{r} 
	& F(Y)   \\
\cdots \arrow{r} 
	& F_3(X,X,X) \arrow{r} \arrow{u}
	& F_2(X,X) \arrow{r} \arrow{u}
	& F(X)  \arrow{u} 
\end{tikzcd}
\]
with $F(Y)$ in degree $(0,0)$. The map $Y\to Z$ induces a map $F(Y) \to F(Z)$ and one obtains an augmented double complex
\begin{equation}\label{eq:double-resolved}
\begin{tikzcd}
\cdots \arrow{r} 
	& F_3(X,X,Y) \arrow{r} 
	& F_2(X,Y) \arrow{r} 
	& F(Y)   \arrow{r}
	& F(Z) \\
\cdots \arrow{r} 
	& F_3(X,X,X) \arrow{r} \arrow{u}
	& F_2(X,X) \arrow{r} \arrow{u}
	& F(X)   \arrow{r} \arrow{u}
	& 0 \arrow{u} 
\end{tikzcd}
\end{equation}
functorial in the short exact sequence $0 \to X \to Y \to Z \to 0$.

\begin{theorem}\label{thm:resolution}
The map $\Tot C \to F(Z)$ induced by (\ref{eq:double-resolved})
is a quasi-isomorphism.
\end{theorem} 

This follows from \cite[Lemma 2.2]{Kock01}, where more generally it is shown that for 
every map $X\to Y$ the complex $\Tot C$ computes the total derived functor of $F$ applied to $X\to Y$ (at least if $X$ and $Y$ are projective). Since Theorem \ref{thm:resolution} concerns a simple case which can be proven and stated without reference to derived non-additive functors, we give a direct proof.

\begin{proof}[Proof of Theorem \ref{thm:resolution}]
We need to show that the total complex associated to the double complex (\ref{eq:double-resolved})
is exact. Choose a splitting $Z\to Y$ of the short exact sequence $0\to X\to Y \to Z\to 0$. In particular, we obtain for every $n$ an inclusion
\[
	F_n(X,\ldots,X,X) \oplus F_n(X,\ldots,X,Z) \injto F_n(X,\ldots,X,Y),
\]
and by Proposition \ref{prop:cross-decomposition} and \ref{prop:cross-associative}, the cokernel is 
 $F_{n+1}(X,\ldots,X,X,Z)$.
We use these inclusions to produce an increasing filtration $\Fil_\bullet$ on the double complex (\ref{eq:double-resolved}) by letting $\Fil_n$ be the sub-double complex
\[
\begin{tikzcd}
\cdots \arrow{r} 
	& 0 \arrow{r} 
	& F_n(X,\ldots,X) \oplus F_n(X,\ldots,Z) \arrow{r} 
	& F_{n-1}(X,\ldots,Y) \arrow{r} 
	& \cdots \\
\cdots \arrow{r} 
	& 0 \arrow{r} \arrow{u}
	& F_n(X,\ldots,X) \arrow{r} \arrow{u}
	& F_{n-1}(X,\ldots,X) \arrow{r} \arrow{u}
	& \cdots 
\end{tikzcd}
\]
The intermediate quotient $\gr_n= \Fil_{n}/\Fil_{n-1}$ takes the form
\[
\begin{tikzcd}
F_n(X,\ldots,X) \oplus F_n(X,\ldots,Z) \arrow{r} & F_n(X,\ldots,Z) \\
F_n(X,\ldots,X) \arrow{u} & 
\end{tikzcd}
\]
where one computes that the maps are the obvious inclusion and projection maps. In particular, the graded quotients of the total complex are exact and hence the total complex itself is exact.
\end{proof}

\section{A presentation of the Grothendieck group}\label{sec:presentation}

Let $\cA_0$ be a weak Serre subcategory of $\cA$. This implies that the category $\Ch_{\geq 0}^{\cA_0} \cA$  of chain complexes $X_\bullet$ in $\cA$ satisfying
\begin{enumerate}
\item $X_i=0$ for all $i<0$,
\item $\rH_i(X_\bullet)=0$ for all $i\gg 0$,
\item $\rH_i(X_\bullet)\in \cA_0$ for all $i$
\end{enumerate}
is an abelian subcategory of $\Ch \cA$. Assume that $\cA$ has enough projectives. Let $\cP \subset \cA$ be the full additive subcategory consisting of all projectives.
Let $\Ch^{\cA_0}_{\geq 0} \cP$ be the full subcategory consisting of those $P_\bullet \in \Ch_{\geq 0}^{\cA_0} \cA$ such that $P_i \in \cP$ for all $i$. Again, every $P_\bullet \in \Ch_{\geq 0}^{\cA_0} \cA$ has an Euler characteristic $\chi(P_\bullet) \in \rK_0(\cA_0)$.

Let  $\sim$ be the equivalence relation on (the set of isomorphism classes of)
$\Ch^{\cA_0}_{\geq 0} \cP$ generated by the relations:
\begin{enumerate}
\item $P_\bullet \sim Q_\bullet$ if $P_\bullet$ and $Q_\bullet$ are homotopy equivalent;
\item $P_\bullet \sim Q_\bullet$ if there exist short exact sequences
$0 \to X_\bullet \to Y_\bullet \to P_\bullet \to 0$ and
$0 \to X_\bullet \to Y_\bullet \to Q_\bullet \to 0$ in $\Ch^{\cA_0}_{\geq 0} \cP$.
\end{enumerate}
Note that, since the complexes consist of projective objects, quasi-isomorphic complexes are always homotopy equivalent.
Clearly if $P_\bullet \sim Q_\bullet$ then $\chi(P_\bullet)=\chi(Q_\bullet)$  in $\rK_0(\cA_0)$.

\begin{proposition}\label{prop:presentation}
The map $\chi \colon (\Ch^{\cA_0}_{\geq 0} \cP)/\!\sim\,\longto\, \rK_0(\cA_0)$ is a bijection.
\end{proposition}

\begin{proof}
We denote the mapping cone of a morphism $X_\bullet\to Y_\bullet$ by $\cone(X_\bullet\to Y_\bullet)$ and note that there is a short exact sequence
\[
	0 \longto Y_\bullet \overset{\alpha}{\longto} \cone(X_\bullet \to Y_\bullet) \overset{\beta}{\longto} X_\bullet[-1] \longto 0.
\]

Let $\cS=(\Ch^{\cA_0}_{\geq 0} \cP)/\!\sim$ be the set of equivalence classes, and let us denote the equivalence class of an $X_\bullet$  by $\{X_\bullet\} \in \cS$. The operation
\[
	\{X_\bullet\} + \{Y_\bullet\} := \{X_\bullet\oplus Y_\bullet\}
\]
is well-defined and makes $\cS$ into a monoid. 

We claim that $\cS$ is even a group. Indeed, comparing the short exact sequences
\[
	0 \longto X_\bullet 
	\overset{}{\longto} \cone(X_\bullet \overset{\id}{\to} X_\bullet) \oplus X_\bullet 
	\longto \cone(X_\bullet \overset{\id}{\to} X_\bullet) \longto 0
\]
and
\[
	0 \longto X_\bullet 
	\overset{(\alpha,0)}{\longto} \cone(X_\bullet \overset{\id}{\to} X_\bullet) \oplus X_\bullet
	\overset{\beta \oplus \id}{\longto} X_\bullet[-1] \oplus X_\bullet \longto 0
\]
we see that $X_\bullet \oplus X_\bullet[-1] \sim \cone(X_\bullet \overset{\id}{\to} X_\bullet)$, and since
$\cone(X_\bullet \overset{\id}{\to} X_\bullet)$ is homotopy-equivalent to $0$ we find 
 $\{X_\bullet\} + \{X_\bullet[-1]\}=0$ in $\cS$. 

Now if $0\to X_\bullet \to Y_\bullet \to Z_\bullet \to 0$ is a short exact sequence, then we claim that $\{Y_\bullet\}=\{X_\bullet\} + \{Z_\bullet\}$. 
Indeed, we have a quasi-isomorphism $\gamma\colon X_\bullet[-1] \to \cone(Y_\bullet \to Z_\bullet)$, and comparing
the exact sequences
\[
	0 \longto  X_\bullet[-1] \oplus Z_\bullet  
	\overset{\scalebox{0.6}{$\left(\begin{array}{cc} \id & \gamma \\ 0 & \alpha \end{array}\right)$}}{\longto}
		X_\bullet[-1]\oplus\cone(Y_\bullet \to Z_\bullet)  
	\longto Z_\bullet[-1] \oplus X_\bullet[-1] \longto 0
\]
and
\[
	0 \longto Z_\bullet \oplus X_\bullet[-1]  
	\overset{\scalebox{0.6}{$\left(\begin{array}{cc} \id & 0 \\ 0 & \alpha \end{array}\right)$}}{\longto}
		X_\bullet[-1]\oplus\cone(Y_\bullet \to Z_\bullet)  
	\longto Y_\bullet[-1] \longto 0	
\]
shows that $Y_\bullet[-1] \sim X_\bullet[-1] \oplus Z_\bullet[-1]$ and hence $\{Y_\bullet\} = \{X_\bullet\} + \{Z_\bullet\}$ in $\cS$. 

Taking projective resolutions of objects in $\cA_0$ defines an injective homomorphism $\psi\colon\rK_0(\cA_0) \to \cS$. To see that it is surjective, we use induction on the amplitude of a complex. We say that $X_\bullet$ has amplitude $\leq a$ if there is an $n$ so that $\rH_i(X_\bullet)=0$ for all $i<n$ and $i\geq n+a$. If $X_\bullet \in \Ch_{\geq 0}^{\cA_0} \cP$ has amplitude $\leq 1$ then up to shift $X_\bullet$ is a projective resolution of an object in $\cA_0$, and lies in the image of $\psi$. If $a\geq 2$ and $X_\bullet$ has amplitude $\leq a$, then for suitable $m$ the ``good truncation'' $\tau_{<m} X_\bullet$ (with homology in $\cA_0$, but necessarily consisting of projectives) gives a short exact sequence
\[
	0 \to \tau_{<m}X_\bullet \to X_\bullet \to \tau_{\geq m} X_\bullet \to 0
\]
in $\Ch_{\geq 0}^{\cA_0} \cA$ with $\tau_{<m}X_\bullet$ and $\tau_{\geq m}X_\bullet$ of amplitude $\leq a-1$. This sequence is quasi-isomorphic with a short exact sequence
\[
	0 \to U_\bullet \to X'_\bullet \to V_\bullet \to 0
\]
consisting of complexes of projectives. Since $X_\bullet$ and $X'_\bullet$ consist of projectives they are homotopy-equivalent, and since $U_\bullet $ and $V_\bullet$ have amplitude $\leq a-1$ we conclude that $\{X_\bullet\} = \{U_\bullet\} + \{V_\bullet\}$ lies in the image of $\psi$.
\end{proof}

\section{Derived functors of non-additive functors} \label{sec:DoldPuppe}

For an abelian category $\cA$ we denote by $\Simp \cA$ the category of simplicial objects in $\cA$ and by $\HoSimp \cA$ its homotopy category (whose objects are the objects of $\Simp \cA$, and whose morphisms are the homotopy classes of morphisms). Any functor $F\colon \cA \to \cB$ induces a functor $\Simp \cA \to \Simp \cB$ which is compatible with simplicial homotopy, and hence induces a functor $\HoSimp \cA \to \HoSimp \cB$.

A simplicial object $X_\bullet \in \Simp \cA$ gives a chain complex
\[
	 \cdots \to X_2 \to X_1 \to X_0 \to 0 
\]
in the usual way, and this induces a functor $C\colon \HoSimp \cA \to \cK_{\geq 0} \cA$. We will use the following variant of the Dold-Kan theorem.

\begin{theorem}\label{thm:Dold-Kan}
The functor  $C\colon \HoSimp \cA \to \cK_{\geq 0} \cA$ is an equivalence of categories. 
\end{theorem}

\begin{proof}This is the Dold-Kan theorem \cite[8.4.1]{Weibel94}, except that we use the full associated chain complex $C(X_\bullet)$ instead of the normalized complex $N(X_\bullet)$. It is not hard to show that the inclusion $N(X_\bullet) \subset C(X_\bullet)$ is a homotopy equivalence. See also \cite[III.2.4]{GoerssJardine}.
\end{proof}

Let $\cP \subset \cA$ be the additive subcategory of projectives, and assume $\cA$ has enough projectives. By the Dold-Kan theorem for every $X_\bullet \in \Ch_{\geq 0} \cA$ 
 there exists a $P_\bullet \in \Simp \cP$ with a quasi-isomorphism $\alpha\colon s(P_\bullet) \to X_\bullet$, and $(P_\bullet, \alpha)$ is unique up to a unique simplicial homotopy equivalence. This construction defines a `simplicial projective resolution' functor
\[
	\rho\colon \cK_{\geq 0} \cA \to \HoSimp \cP,
\]
which is the essential ingredient in the definition of derived functors of non-additive functors.

\begin{definition}
Let $F\colon \cA\to \cB$ be a functor. Then the composition
\[
	\cK_{\geq 0} \cA \overset{\rho}{\longto} \HoSimp \cP \overset{F}{\longto} \HoSimp \cB
	\overset{s}{\longto} \cK_{\geq 0} \cB
\]
is called the \emph{total derived functor} of $F$, and denoted $\rL F$.
\end{definition}

Finally, let $\Simp^{\cA_0} \cP$ be the full subcategory of $\Simp \cP$ consisting of those simplicial objects $P_\bullet$ with bounded homology contained in $\cA_0$. Let $\sim$ be the equivalence relation on $\Simp^{\cA_0} \cP$ generated by
\begin{enumerate}
\item $P_\bullet \sim Q_\bullet$ if $P_\bullet$ and $Q_\bullet$ are homotopy equivalent,
\item $P_\bullet \sim Q_\bullet$ if there exist short exact sequences $0 \to X_\bullet \to Y_\bullet \to P_\bullet \to 0$ and $0 \to X_\bullet \to Y_\bullet \to Q_\bullet \to 0$ in $\Simp^{\cA_0} \cP$.
\end{enumerate}

\begin{corollary}\label{cor:simplicial-presentation}
The map $\chi\colon (\Simp^{\cA_0} \cP)/\!\sim\, \longto \rK_0(\cA_0)$ is a bijection.
\end{corollary}

\begin{proof}This follows from Proposition \ref{prop:presentation} and Theorem \ref{thm:Dold-Kan}.
\end{proof}

\section{Proof of the main result}

Let $\cA_0$ be a weak Serre subcategory of $\cA$.  Assume that every $X\in \cA_0$ has a projective resolution in $\cA$. Let $\cP \subset \cA$ be the exact category of all projectives. Let $\Simp^{\cB_0} \cB$ be the full subcategory of $\Simp \cB$ consisting of those simplicial objects with bounded homology contained in $\cB_0$.

\begin{theorem}
Let $d$ be a positive integer.
Let $F\colon \cA \to \cB$ be a functor with $F(0)=0$. Assume $F$ is of degree $\leq d$
and  that it maps $\Simp^{\cA_0} \cP$ to $\Simp^{\cB_0} \cB$. Then there exists a unique map $f\colon \rK_0(\cA_0) \to \rK_0(\cB_0)$ such that the square
\[
\begin{tikzcd}
\Simp^{\cA_0} \cP \arrow{r}{F} \arrow{d}{\chi} & \Simp^{\cB_0} \cB \arrow{d}{\chi} \\
\rK_0(\cA_0) \arrow{r}{f} & \rK_0(\cB_0)
\end{tikzcd}
\]
commutes. Moreover, the map $f$ is of degree $\leq d$.
\end{theorem}

This theorem implies Theorem \ref{thm:one} of the introduction.

\begin{proof}
By Corollary \ref{cor:simplicial-presentation}, the map $f$ is unique, and to establish existence it suffices to show that for every
$P_\bullet$, $Q_\bullet$ in $\Simp^{\cA_0} \cP$ we have
that  $P_\bullet \sim Q_\bullet$ implies $\chi(F(P_\bullet))=\chi(F(Q_\bullet))$
in $\rK_0(\cB_0)$. Recall  that the equivalence relation $\sim$ is generated by homotopy equivalences and by relations coming from short exact sequences.

If $P_\bullet$ and $Q_\bullet$ are simplicially homotopy-equivalent, then $F(P_\bullet)$ and $F(Q_\bullet)$ are simplicially homotopy-equivalent, and hence $\chi(F(P_\bullet))=\chi(F(Q_\bullet))$ in $\rK_0(\cB_0)$. 

Now let
\[
	0 \to X_\bullet \to Y_\bullet \to P_\bullet \to 0
\]
be a short exact sequence in $\Simp^{\cA_0} \cP$. Note that for every $n$ the sequence $0\to X_n \to Y_n \to P_n\to 0$ is split  (since $P_n$ is projective), but that the sequence in $\Simp^{\cA_0} \cP$ need not split. For every $n$ the construction of \S \ref{sec:Koeck-resolution} gives a double complex
\[
\begin{tikzcd}
\cdots \arrow{r} 
	& F_3(X_n,X_n,Y_n) \arrow{r} 
	& F_2(X_n,Y_n) \arrow{r} 
	& F(Y_n)   \\
\cdots \arrow{r} 
	& F_3(X_n,X_n,X_n) \arrow{r} \arrow{u}
	& F_2(X_n,X_n) \arrow{r} \arrow{u}
	& F(X_n)  \arrow{u} 
\end{tikzcd}
\]
whose total complex is a resolution of $F(P_n)$.
 Since the construction of the double complex is functorial in the map $X\to Y$, we  obtain a double complex 
\[
\begin{tikzcd}
\cdots \arrow{r} 
	& F_3(X_\bullet,X_\bullet,Y_\bullet) \arrow{r} 
	& F_2(X_\bullet,Y_\bullet) \arrow{r} 
	& F(Y_\bullet)   \\
\cdots \arrow{r} 
	& F_3(X_\bullet,X_\bullet,X_\bullet) \arrow{r} \arrow{u}
	& F_2(X_\bullet,X_\bullet) \arrow{r} \arrow{u}
	& F(X_\bullet)  \arrow{u} 
\end{tikzcd}
\]
in $\Simp \cB$, whose associated total complex is a resolution of $F(P_\bullet)$. Because $F$ is of finite degree, this is a \emph{finite} resolution.

Each of the terms is a direct summand of a simplicial object of the form
\[
	F(X_\bullet\oplus \cdots \oplus X_\bullet)\quad\text{or}\quad
	F(X_\bullet\oplus \cdots \oplus X_\bullet \oplus Y_\bullet),
\]
and hence  lies in $\Simp^{\cB_0} \cB$. This means that in $\rK_0(\cB_0)$ we have  
\[
	\chi(F(P_\bullet)) = \sum_{n=1}^d (-1)^n \Big(
		\chi\big(F_n(X_\bullet,\ldots, X_\bullet)\big) -
		\chi\big(F_n(X_\bullet,\ldots, Y_\bullet)\big) \Big).
\]
In particular, since the terms do not depend on the map $X_\bullet \to Y_\bullet$, we see that if 
\[
	0 \to X_\bullet \to Y_\bullet \to Q_\bullet \to 0
\]
is a second short exact sequence in $\Simp^{\cA_0} \cP$ then 
$\chi(F(P_\bullet)) = \chi(F(Q_\bullet))$ in $\rK_0(\cB_0)$. This proves the existence of $f$.

Finally, note that the $(d+1)$-st cross effect of
the functor
\[
	F\colon \Simp^{\cA_0} \cP \to \Simp^{\cB_0} \cB
\]
vanishes, which implies the analogous statement for the function $f$, and shows that
$f$ is of degree $\leq d$.
\end{proof}

\end{document}